\documentclass{au}

\usepackage{amsmath}
\usepackage{amssymb,url}
\usepackage{graphicx}

\numberwithin{equation}{section}
\theoremstyle{plain}
\newtheorem{theorem}{Theorem}[section]
\newtheorem{lemma}[theorem]{Lemma}

\newtheorem{corollary}[theorem]{Corollary}

\theoremstyle{definition}
\newtheorem{definition}[theorem]{Definition}

\title{Term inequalities in finite algebras}
\newcommand{\pth}{\mathop{\rm path}}

\newcommand{\A}{{\mathbf A}}
\newcommand{\Alg}{\mathop{\rm Alg}}
\newcommand{\length}{\mathop{\rm length}}
\begin{document}
%%%%%%%%%%%%  FRONT MATTER  %%%%%%%%%%%%%%%%%%%%%%%%%%%%%%%%%%%%%%%%%%%%%%%%%%%%%%%%%%%%%%%%%%%%%

\title[Term inequalities in finite algebras]{Term inequalities in finite algebras}

%% First author (Note: The order of the items here is important!)

\author[D. Hobby]{David Hobby} 

\email{hobbyd@newpaltz.edu}

\address{Mathematics, SUNY New Paltz\\
1 Hawk Drive\\
New Paltz, NY, 12561\\
USA}

%% Second author (Note: The order of the items here is important!)
%\author{}

%% Dedication (Optional)

%\dedicatory{}

%% AMS subject classification; see http://www.ams.org/msc
%% Only one Primary. Possibly several Secondary.

%% Keywords and phrases

\begin{abstract}
Given an algebra \ $\mathbf{A}$, \ and terms $s(x_{1},x_{2},\dots x_{k})$
and $t(x_{1},x_{2},\dots x_{k})$ of the language of $\A$,
we say that $s$ and $t$ are {\em separated} in $\A$ iff for
all \ $a_{1},a_{2}\dots a_{k}\in A$, \ $s(a_{1},a_{2},\dots a_{k})$
\ and \ $t(a_{1},a_{2},\dots a_{k})$ are never equal. We prove
that given two terms that are separated in any algebra, there exists
a finite algebra in which they are separated. As a corollary, we obtain
that whenever the sentence $\sigma$ is a universally quantified conjunction
of negated atomic formulas, $\sigma$ is consistent iff it has a finite
model. 
\end{abstract}

\maketitle

%%%%%%%%%%%%%%%%%%%%  MAIN MATTER %%%%%%%%%%%%%%%%%%%%%%%%%%%%%%%%%%%%%%\centerline{\textbf{{}}}{ \par}

\begin{section}{Introduction}\par Around fifteen
years ago, the author was involved with D. Silberger in an investigation
of finite groupoids which we called antiassociative. Instead of obeying
the associative law that \ $(x_{1}\star x_{2})\star x_{3}$ \ and
\ $x_{1}\star(x_{2}\star x_{3})$ \ were always equal, a groupoid
was {\em antiassociative} iff \ $(x_{1}\star x_{2})\star x_{3}$
\ and \ $x_{1}\star(x_{2}\star x_{3})$ \ were never equal. We
viewed this as a natural change to make to the associative law.

Working with M. Braitt, we generalized this to the problem of finding finite
groupoids that were {\em k-antiassociative}, meaning that all of
the distinct terms made by inserting parentheses in the string $x_{1}\star x_{2}\star x_{3}\star\dots x_{k}$
were never equal. This reduces to the problem of constructing a finite
algebra that separates two distinct terms, where we say that the algebra
\ $\mathbf{A}$, {\em separates} the terms $s(x_{1},x_{2},x_{3},\dots x_{k})$
and $t(x_{1},x_{2},x_{3},\dots x_{k})$ iff for all \ $a_{1},a_{2}\dots a_{k}\in A$,
\ $s(a_{1},a_{2},\dots a_{k})$ \ and \ $t(a_{1},a_{2},\dots a_{k})$
are never equal. This is because a product of algebras separates all
of the pairs of terms separated in any of its factors.

Once we were looking at separating pairs of terms, it was natural to generalize
this to groupoid terms $s$ and $t$ which had their variables appearing
arbitrarily often in arbitrary orders. We solved the problem for many
pairs of groupoid terms $s$ and $t$ in \cite{BHS}, showing
that if $s$ and $t$ were ever separated in any groupoid, they were
separated in a finite one.

Our terminology is standard for modern
universal algebra. The reader is referred to \cite{BurrisSankappanavar}
by Burris and Sankappanavar for undefined terms and notation. Our
algebras will have possibly infinitely many basic operations, all
of finite arity. Constants are allowed, and are $0$-ary operations.
The language of an algebra has symbols for all of its basic operations.
Terms are built up by applying basic operations to variable symbols.
We will use the same notation for terms as expressions in a language
and for the term functions obtained by interpreting terms in an algebra.

It is sometimes useful to adopt the convention that terms may be written
with extra variables that do not actually appear in them, so writing
$s(x_{1},x_{2},\dots x_{k})$ and $t(x_{1},x_{2},\dots x_{k})$ indicates
that the variables used in $s$ and $t$ are subsets of $\{x_{1},x_{2},\dots x_{k}\}$.

Observe that terms $s(x_{1},x_{2},\dots x_{k})$ and $t(x_{1},x_{2},\dots x_{k})$
are separated in free algebras iff they are ever separated in any
algebra. For suppose $\A$ separates $s$ and $t$, while $s(b_{1},b_{2},\dots b_{k})=t(b_{1},b_{2},\dots b_{k})$
for some $b_{1},b_{2},\dots b_{k}$ in the universe of a free algebra
${\mathcal{F}}$. Then letting $\phi\colon{\mathcal{F}}\to\A$, we
have $s(\phi(b_{1}),\phi(b_{2}),\dots\phi(b_{k}))=t(\phi(b_{1}),\phi(b_{2}),\dots\phi(b_{k}))$
in $\A$, a contradiction. Free algebras are usually infinite, so
it is natural to ask if there are finite algebras separating $s$
and $t$.

Let us look at $s(x_{1},x_{2},\dots x_{k})$ and $t(x_{1},x_{2},\dots x_{k})$
in ${\mathcal{F}}={\mathcal{F}}(x_{1},x_{2},\dots x_{k})$, the free
algebra on $x_{1},x_{2},\dots x_{k}$ in the language of $s$ and
$t$. If the terms $s$ and $t$ are not separated in ${\mathcal{F}}$,
then there are $b_{1},b_{2},\dots b_{k}$ in the universe of ${\mathcal{F}}$
where $s(b_{1},b_{2},\dots b_{k})=t(b_{1},b_{2},\dots b_{k})$. In
other words, there are terms $b_{1},b_{2},\dots b_{k}$ so that substituting
$b_{1}$ for $x_{1}$, $b_{2}$ for $x_{2}$, and so on in both $s$
and $t$ turns them both into the same term. This is commonly referred
to as {\em unifying} the terms $s$ and $t$.

The unification problem has been extensively studied in computer science. The introduction
of the topic was by Herbrand, in \cite{Herbrand}. Modern work was
pioneered by Robinson, in \cite{Robinson}. A good survey article is
\cite{BaaderSnyder} by Baader and Snyder. Our proof will proceed
by careful analysis of a deduction system for unification.

Two terms $s$ and $t$ are {\em unifiable}
if they can be unified, and the corresponding substitution of terms
for their variables is a {\em unification}. 
%Although it is uncommon in the literature, there is no harm in 
%continuing our convention that %terms may formally have extra variables that do not 
%actually appear in them; unifying substitutions may assign 
%arbitrary terms to variables that do not appear.
We will prove that whenever terms $s$ and $t$ can be separated in
any algebra, that they can be separated in a finite algebra. We have
that two terms can be separated in any algebra iff they can be separated
in a free algebra iff they can not be unified. So we need to establish
that terms can not be unified iff they are separated in a finite algebra.
One direction of this follows immediately from the above, so it remains
to show that whenever two terms can not be unified, that they are
separated in a finite algebra. Doing this requires a careful analysis
of a unification algorithm, so that we can construct the desired finite algebra.

Algorithms to see whether or not
two terms $s$ and $t$ can be unified are discussed in detail in
\cite{BaaderSnyder}. Since we are not concerned
with efficiency, we will use an algorithm based on unordered application
of deduction rules. (This is called tree-based syntactic unification.)

We will start with a fairly abstract description of unification from
\cite{BaaderSnyder}, and recover a set of rules from it that will
work well in our proofs.  Consider two terms $s(x_1, \dots x_{m})$ 
and $t(y_1, \dots y_{n})$.  
The terms are {\em unifiable} if there are terms 
$r_1,\dots r_{m}$ and $u_1, \dots u_{n}$ so that substituting
the $r_i$ for the $x_i$ in $s$ and the $u_j$ for the $y_j$ in $t$
makes the two resulting terms identical.  This identical term is a {\em unifier} and the corresponding substitution is a {\em unification}.  In other words, the terms
$s$ and $t$ can be unified iff they can not be separated in a free algebra.  

While much work on unification is concerned with algorithms, we will follow the more abstract approach in \cite{BaaderSnyder},
which was first presented by G.P. Huet in \cite{Huet}.
As in its Definition 2.11, we consider equivalence relations on groupoid terms, which we call {\em term relations}.  
A term relation is said to be {\em homogeneous} if no terms of the form 
$f(\dots)$ and $g(\dots)$ are ever equivalent for distinct operation symbols $f$ and $g$.  

Given a term relation $\equiv$, we let $\prec$ be the ``is equivalent to a subterm of'' relation.
That is, $p \prec q$ iff there is some subterm $r$ of $q$ with $p \equiv r$.  We say that a
sequence $p_0 \prec p_1 \prec \dots p_m$ is a {\em $\prec$-cycle} iff $p_m = p_0$ and at least 
one of the subterm relationships is proper.  Then the term relation $\equiv$ is 
{\em acyclic} iff it has no $\prec$-cycles.  
(This follows the definition of `acyclic' used in papers such as \cite{Dwork}.  The definition in
\cite{BaaderSnyder} merely says that a term relation is `acyclic' if no term is equivalent to one of its proper subterms, which is incorrect.)
%We will soon give an example that shows the
%latter definition is incorrect, where the unification closure is actually not acyclic but 
%is by the definition in \cite{BaaderSnyder}.)

This leads to the following definition.

\begin{definition}\label{unification relation def}
A term relation $\equiv$ is a {\em unification relation} iff it is homogeneous, 
acyclic, and satisfies the following
{\em unification axiom}:
Whenever $s \equiv t$ where $s$ is $f(p_1, p_2, \dots p_n)$ and
$t$ is $f(q_1,q_2, \dots q_n)$ for the same $n$-ary basic operation $f$, 
then $p_1 \equiv q_1$, $p_2 \equiv q_2$, \ldots $p_n \equiv q_n$.
\end{definition}

Referring again to \cite{BaaderSnyder} for the details, we have that terms $s$ and $t$ can be unified iff 
there is a unification relation $\equiv$ with $s \equiv t$.  If there is such a unification relation, then there
is a unique minimal one, the {\em unification closure} of $s$ and $t$.  If $s$ and $t$ can be unified, they also have
a {\em most general unifier} or {\em mgu}, where any unifier of $s$ and $t$ can be obtained from their mgu by uniformly 
substituting terms for its variables.  This most general unifier is unique up to renaming its variables, and can be 
easily constructed from the unification closure of $s$ and $t$.

For example, let $\star$ be an infix binary operation, and 
consider $s = (x \star y) \star (z \star y)$ and
$t = z \star ((x \star y) \star (x \star x))$. 
We will attempt to construct 
their unification closure $\equiv$ and the corresponding mgu.  
We must have $s \equiv t$, and start with this.
Using the unification axiom, we obtain $x \star y \equiv z$ and
$z \star y \equiv (x \star y) \star (x \star x)$.  Applying the unification axiom 
again to the last equivalence, we get $z \equiv x \star y$ (a duplicate) and
$y \equiv x \star x$.  The non-singleton classes of $\equiv$ that contain variables are now
$\{ y, x \star x \}$ and $\{ z, x \star y \}$ , where $x$ is in a class by itself.

To construct the mgu, we pick a representative of each class, where we must pick 
a non-variable term if there is one in the class.  Letting $u$, $v$ and $w$ be 
arbitrary terms, we let $\varsigma(w)$ be the representative 
of the class of $w$.  In our example, this gives $\varsigma(y) = x \star x$,
$\varsigma(z) = x \star y$, and $\varsigma(x) = x$.  
Now we recursively define the function $\sigma$ 
from terms to terms by letting $\sigma(w)$ be 
$\varsigma(w)$ if $\varsigma(w)$ is a variable, and 
letting $\sigma(w)$ be $\sigma(u) \star \sigma(v)$
if $\varsigma(w)$ is $u \star v$.  In our example, this gives
$\sigma(s) = \sigma(x \star y) \star \sigma(z \star y) = 
(\sigma(x) \star \sigma(y)) \star (\sigma(z) \star \sigma(y))
= (x \star (\sigma(x) \star \sigma(x))) \star ((\sigma(x) \star \sigma(y)) 
\star (\sigma(x) \star \sigma(x)))
= (x \star (x \star x)) \star ((x \star (x \star x)) \star (x \star x))$, 
where the last term is the mgu of $s$ and $t$.       
The reader may check that applying $\sigma$ to $t$ yields the same result.

Here is an example where terms can not be unified because there is a 
$\prec$ cycle.  Let $f$ be a ternary operation, and let $g$ be binary.
Let $s = f(x,g(u,v),y)$ and $t = f(g(y,w),g(x,z),g(u,v))$.  
We start with $s \equiv t$, and get $x \equiv g(y,w)$, $g(u,v) \equiv g(x,z)$
and $y \equiv g(u,v)$ by the unification axiom.  Since $\equiv$ is transitive,
we get $y \equiv g(x,z)$, which makes $g(x,z)$ equivalent to $y$, which is a
subterm of $g(y,w)$, yielding $g(x,z) \prec g(y,w)$.  Similarly,
$g(y,w) \equiv x$ and $x$ is a subterm of $g(x,z)$, so $g(y,w) \prec g(x,z)$.
The $\prec$ cycle $g(x,z) \prec g(y,w) \prec g(x,z)$ means that there is no 
unification relation for $s$ and $t$, showing they can not be unified.
(However, $\equiv$ does not have to be a congruence, 
%We can not substitute
%$y \equiv g(x,z)$ into $x \equiv g(y,w)$ to get $x \equiv g(g(x,z),w)$, and 
so no term is equivalent to one of its proper subterms.)

Now for our deduction system. These rules
generate lists of statements, $D$, where each statement comes from some 
previous statements in the list.  We will use $\equiv$ instead of $=$ in 
our statements, and start $D$ with the single statement $s \equiv t$, 
where $s$ and $t$ are terms to be unified.
(We will view $s \equiv t$ and $t \equiv s$ as the same statement, 
building in that the relation $\equiv$ is symmetric.)

In each rule, $a$, $a_{1},a_{2},a_{3}\dots$, and $b_{1},b_{2},b_{3}\dots$, and so on
are terms, $x$ and $y$ are variables, and $f$ and $g$ are function
symbols. (Function symbols may also denote constants, and will not
be listed with extra arguments that do not appear.)
We will never need to deduce statements of the form $a \equiv a$, 
so to use that $\equiv$ is an equivalence relation on terms, it is enough
to require that it be transitive.  
We will also treat nested applications of the 
unification axiom as a single step which deduces that subterms in corresponding locations are equivalent.  

To make the location of a subterm more explicit, 
we give the following definition.  
Note that a term $t$ mave have the same term $u$ occur as a 
subterm in different locations, so we should really refer to
{\em occurrences} of subterms.
If $p$ is any term and $q$ is an occurrence of 
a subterm of $p$ , then we define the {\em path}
of $q$ in $p$, $\pth(q)$, to be a string of subscripted operation
names which we define recursively as follows.  If $p$
is a constant or a single variable $x_{i}$, then that constant or
$x_{i}$ is the only subterm of $p$ and its path is the empty string
$\Lambda$. If $p$ is $f(q_{1},q_{2},\dots q_{n})$
where $f$ is an $n$-ary operation and the $q_{i}$ are occurrences 
of terms, and $r$ is an occurrence of a subterm of $q_{i}$ 
with path $\rho$, then that occurrence of $r$ is a subterm
of $p$ with path $f_{i}\rho$, where $f_{i}\rho$ is the concatenation
of $f_{i}$ and $\rho$. We let $\mathit{P}$ be the set of all possible
paths for the particular set of basic operations we are using.
Thus occurrences of subterms of two terms are in corresponding 
locations iff they have the same path in their respective terms.
For simplicity, we will henceforth refer to occurrences of 
subterms as just subterms.

There are two ways that an attempt at unification can 
fail, which we will denote by $\mathbf{False}$.  So we will have rules to 
deduce $\mathbf{False}$ from a failure of $\equiv$ to be homogeneous, and also from
a failure to be acyclic.  This gives us the following four rules.

\begin{enumerate}

\item {\em (Transitive)}  For any $n > 2$, from $a_1 \equiv a_2$, $a_2 \equiv a_3$, \ldots 
$a_{n-1} \equiv a_n$, deduce $a_1 \equiv a_n$.

\item {\em (Decompose)} From $a \equiv b$, 
deduce $c \equiv d$ whenever $c$ is a subterm of $a$ with path $\rho$,
and $d$ is a subterm of $b$ with the same path $\rho$. 

\item {\em (Conflict)} From $f(a_{1},\dots a_{n})\equiv g(b_{1},\dots b_{m})$ with $f \neq g$, deduce $\mathbf{False}$.

\item {\em (Cycle)} For any $n \geq 1$, given 
$p_1 \equiv q_1, p_2 \equiv q_2 \ldots p_{n} \equiv q_{n}$, 
where $q_1$ is a subterm of $p_2$, $q_2$ is a subterm of $p_3$, and so on up to
$q_{n}$ being a subterm of $p_{n+1} = p_1$, forming a $\prec$-cycle, 
deduce $\mathbf{False}$. 

\end{enumerate} 

One may simply apply all the rules repeatedly, until
no more statements are deduced. If $\mathbf{False}$ is ever deduced,
the original terms $s$ and $t$ can not be unified. Otherwise, a
unifying set of substitutions will be deduced. 

Our goal is to show that whenever $\mathbf{False}$
can be deduced from $s\equiv t$, that $s$ and $t$ can be separated in
a finite algebra. In the next section, we will develop general tools
for constructing these algebras.

Let a {\em deduction} be a list of statements starting with $s \equiv t$,
where each statement can be obtained from the set of previous
statements by a single application of one of the rules {\em Transitive},
{\em Decompose}, {\em Conflict} or {\em Cycle}.  We will sometimes 
add parenthetical explanations when writing inductions.

Call a deduction ending in a statement $\sigma$
{\em minimal} iff no statement in the deduction can be removed to 
yield a shorter deduction of $\sigma$.

For example, consider $s = f(g(y,z),g(y,x),x)$ and $t = f(x,g(x,z),g(y,z))$.  
We have $g(y,z) \equiv x$, $g(y,x) \equiv g(x,z)$ and $x \equiv g(y,z)$ by {\em Decompose}.
We view $g(y,z) \equiv x$ and $x \equiv g(y,z)$ as identical, so one of them is redundant. 
From $g(y,x) \equiv g(x,z)$, we get $y \equiv x$ and $x \equiv z$
by {\em Decompose}, either of which can be 
used with $g(y,z) \equiv x$ and {\em Transitive}
to allow {\em Cycle} to deduce $\mathbf{False}$.

The deduction $\langle s \equiv t, g(y,z) \equiv x, g(y,x) \equiv g(x,z), x \equiv g(y,z), 
y \equiv x, x \equiv z, g(y,z) \equiv y, \mathbf{False} \rangle$ 
is not minimal, but it can be reduced to minimal deductions
of $\mathbf{False}$ such as 
$\langle s \equiv t, g(y,z) \equiv x, x \equiv y 
\mbox{(on the path $f_2g_1$)},
g(y,z) \equiv y, \mathbf{False} \rangle$ or
$\langle s \equiv t, x \equiv z \mbox{(on the path $f_2g_2$)},
x \equiv g(y,z), g(y,z) \equiv z, \mathbf{False} \rangle$, 
both of which are minimal.

\section{Tools for constructing algebras}

We will use a somewhat involved construction, and will
require some preliminary definitions.  Recall that the 
{\em path} of a subterm was defined in the previous section.

To help visualize terms, we can represent them as
rooted trees where interior nodes are labeled with operations and
leaves are labeled with variables or constants. For example, suppose
$f$ is a ternary operation, $g$ is binary, $c$ is a constant, and
$u$, $v$, $w$, $x$, $y$ and $z$ are variables. We let $s=f(g(u,v),f(w,x,f(u,v,w)),c)$
and let $t=f(g(v,v),f(w,w,g(y,z)),c)$ . Then the subterm of $s$
with path $f_{1}g_{1}$is $u$, and the subterm of $t$ with path
$f_{1}g_{1}$is $v$.

\begin{figure}
\includegraphics[scale=0.85]{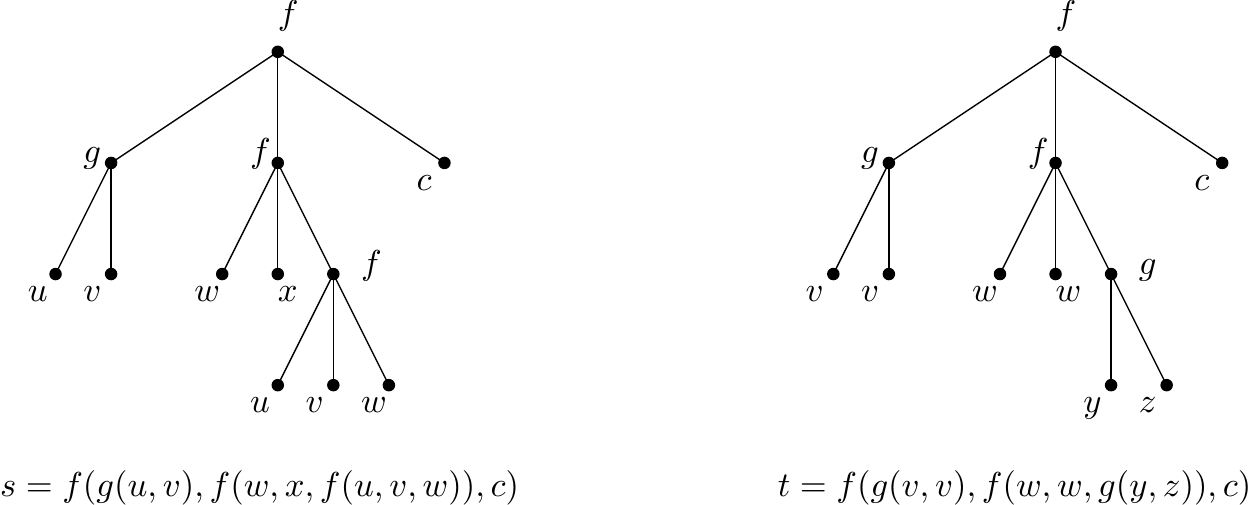}\caption{Trees for two terms}

\end{figure}

Note that the subterms of $s$ and $t$ with path $f_{2}f_{3}$ are
$f(u,v,w)$ and $g(y,z)$ respectively. Since these subterms have
different principal operation symbols, the terms $s$ and $t$ can
not be unified. Applying {\em Decompose}
to $s=t$ with path $f_2f_3$ gives $f(u,v,w) \equiv g(y,z)$, after 
which {\em Conflict} gives $\mathbf{False}$.

Our long-term goal is to form an algebra that
separates any two non-unifiable terms $s$ and $t$, such as those
in our example. We will need some preliminary ideas in order to do
this. Our algebras will have elements which are vectors 
over the $2$-element field $Z_{2}$. We take the index set
of the components of these vectors to be the set of natural numbers
$\mathbf{N}=\{0,1,2,\dots\}$. All of our vectors will be  
zero in all but finitely many components.
Given any finite set of such vectors, we let $M$ be the set of all
indices where any of the vectors is
nonzero. Then all these vectors lie
in the finite subspace consisting of vectors with all their components
outside of $M$ equal to $0$. We will usually leave this final reduction
to a finite algebra to the reader. 

We will actually be only using
the additive structure of the field $\mathbf{Z}_{2}$ , and viewing
it as an abelian group. The operations of our algebras will be sums
of linear transformations of the input vectors, sometimes with constant
elements added. Since we are working over $\mathbf{Z}_{2}$, all additions
of values will be done modulo $2$. We will periodically note this
fact, but not always. We will simply write $x_{i}$ for a vector
variable instead of $\vec{x_{i}}$, and write $x_{i}[a]$ for the
$a$-th component of the vector $x_{i}$. We will define operations
by their actions on components, and
can treat any term $t$ as a vector which denotes the value of
$t(x_1,x_2,\dots,x_n)$.  To specify the operation $f(x_{1,}x_{2},\dots x_{n})$,
it then suffices to say what $f[i]$ is for all $i$. We will do this
by giving a sequence of equations for the $f[i]$. To emphasize that
values are being assigned to the components $f[i]$, we will use $:=$
instead of the normal equality symbol. One further convention is that
each $f[i]$ will be zero, unless it is explicitly assigned a value.
With this convention, each equation corresponds to a linear transformation
from the direct product of the input vector spaces to the output vector
space, where $f[i]:=x_{j}[k]$ corresponds to the transformation that
takes the $n$-tuple of vectors
$\langle x_{1},x_{2},\dots x_{n}\rangle$ to the vector with
all components $0$ except that the $i$-th component is equal to
the $k$-th component of the vector $x_{j}$.

As an example, we will now construct an algebra
that separates our terms 
$s=f(g(u,v),f(w,x,f(u,v,w)),c)$ and $t=f(g(v,v),f(w,w,g(y,z)),c)$.
The separation will be assured because the output vectors $s$ and
$t$ will always differ on their $0$-th component, which we will arrange
as follows. The crucial difference between $s$ and $t$ is that their
subterms with path $f_{2}f_{3}$ have principal operations $f$ and
$g$, respectively. So we start our definition of $f$ by putting
the vector $\langle 0,0,0,0,\dots\rangle$ into the sum that defines
it, while putting $\langle 0,0,1,0,\dots\rangle$ into the sum that
defines $g$. So far, this makes the subterms $f(u,v,w)$ and $g(y,z)$
differ on their $2$nd component. Now we have to transfer this distinction
to the $0$-th components of $s$ and $t$, along the path $f_{2}f_{3}$.
This first application of $f$ comes via the $3$rd input, so we need
to have $f$ transfer the value in the $2$nd component to an unused
component, say the $1$st. This gives us the linear transformation
with equation $f[1]:=x_{3}[2]$. Since this transfers the value in
the $2$nd component to the $1$st component along the path $f_{3}$,
our notation for this linear transformation will be $\|2,f_{3},1\|$.
The next application of $f$ is to the $2$nd input, where we want
to take the value in the $1$st component and send it to the $0$-th
component. The equation for this is $f[0] := x_{2}[1]$, and our notation
for this linear transformation is $\| 1,f_{2},0 \|$. Adding together
the pieces we have produced, we have that 
$f=\langle 0,0,0,\dots \rangle + \|2,f_{3},1\| + \|1,f_{2},0\|$,
while $g$ is just $\langle 1,0,0,0,\dots\rangle$. 

To produce a finite
algebra that separates $s$ and $t$ , we note that only components
$0$, $1$ and $2$ are used, so the universe of our algebra can be
$\mathbf{Z}_{2}^{3}$. We define $g$ by $g(x,y) = \langle 1,0,0\rangle$,
and $f$ by 
$f(x,y,z) = \langle 0,0,0\rangle + \langle 0,0,z(0)\rangle + 
\langle 0,y(2),0 \rangle = \langle 0,y(2),z(0)\rangle$.
The constant $c$ was not used, so we can assign an arbitrary value
to it, and choose $c = \langle 0,0,0 \rangle$. To confirm that this works,
we compute $s[0]$ and $t[0]$. We have $s=f(g(u,v),f(w,x,f(u,v,w)),c)$,
so $s[0]$ gets its value from $f$, actually from $\| 1,f_{2},0 \|$,
which is the only piece of $f$ that assigns a value to the $0$-th
component. Thus $s[0]$ is $f(w,x,f(u,v,w))[1]$. To find the value
of this, we use that the only piece of $f$ assigning a value to the
$1$st component is $\| 2,f_{3},1 \|$, and get 
$f(w,x,f(u,v,w))[1] = f(u,v,w)[2]$.
No piece of $f$ assigns a value to the $2$nd component, so $f(u,v,w)[2] = 0$,
the default value. Thus $s[0] = f(w,x,f(u,v,w))[1] = f(u,v,w)[2] = 0$.
Similarly, $t=f(g(v,v),f(w,w,g(y,z)),c)$ gives us 
$t[0] = f(w,w,g(y,z))[1] = g(y,z)[2] = 1$.

Note that we can treat the sum $\| 2,f_{3},1 \| + \| 1,f_{2},0 \|$ in $f$
as a single transformation, that takes the $2$nd component of the
subterm with path $f_{2}f_{3}$ to the $0$-th component of the main
term. We write this as $\| 2,f_{2}f_{3},0 \|$, which does not mention
the use of the $1$st component in an intermediate step. Since all
we needed was that the component not be used elsewhere, this is reasonable.
So we will assume that indices such as $a$, $b$
and so on are always chosen to minimize {\em collisions}. This
means that no indices will be equal unless they are explicitly represented
with equivalent expressions. This can be easily achieved by appropriate
choices of values for the indices, and will not jeopardize the finiteness
of any algebras we produce. As long as there are no collisions, algebras
obtained for different values of $a$ will be isomorphic. Accordingly,
we will speak of {\em the} algebra operation $\|2,f_{2}f_{3},0\|$,
 etc.

Extending this idea, we can define the transformation $\| j,\rho, k\|$
that takes the $j$-th component to the $k$-th component along path
$\rho$, where $\rho$ is any non-empty string.

\begin{definition}\label{basic operation def}
If $\rho$ is
$f_{i}$ for some operation symbol $f = f(x_{1},\dots,x_{i},\dots,x_{n})$,
we define $\| j,f_{i},k \|$ by the equation 
$f[k] := x_{i}[j]$.
If $\rho$ is $f_{i}\sigma$ where $\sigma\neq\Lambda$, we define
$\| j,\rho,k \|$ to be $\| j,\sigma,m \| + \| m,f_{i},k \|$, where $m$
is understood to be an index not used elsewhere.
\end{definition}

This definition produces a sum of transformations for possibly many
different operations, which is not a problem. To recover the actual
operations from a sum of transformations, $\Sigma$,
we merely let each operation
be the sum of the transformations and constant vectors that reference
that operation, where operations that are not referenced have constant
value $\langle0,0,0,\dots\rangle$.
We denote the algebra with operations defined this way by 
$\Alg(\Sigma)$.  Our understanding is that only indices with non-zero
components are used in $\Alg(\Sigma)$, so it is finite
and unique up to isomorphism.

The idea is that \ $\|m,\rho,n\|$ \ transfers the
value of the $m$-th component of the vector with path $\rho$ in the
term $s$ to the $n$-th component of the result of $s$, with as
few side effects as possible. We are assuming that none of the indices
used to define \ $\|m,\rho,n\|$ \ is equal to any of the others, except
that possibly \ $m=n$. \ In other words, the operation \ $\|m,\rho,n\|$
\ is {\em duplicate free}. If $m_{1}$ is distinct from both $m_{2}$
and $m_{0}$, and $\mu$ and $\nu$ are strings in \ $\mathit{P}$, \ then
the operation \ $\|m_{2},\nu,m_{1}\|+\|m_{1},\mu,m_{0}\|$ \ is duplicate
free by our convention that indices are chosen to minimize collisions.
In isolation, the sum \ $\|m_{2},\nu,m_{1}\|+\|m_{1},\mu,m_{0}\|$ \ is
equivalent to \ $\|m_{2},\mu\nu,m_{0}\|$. \ The one difference is that
the former explicitly mentions the index $m_{1}$. 

\begin{lemma}\label{path lemma}
Let the algebra operation \ $\|m,\rho,n\|$ \ be duplicate free where 
$\rho \in P$.
Suppose that $t$ is a term, and let $s$ be a subterm of $t$ with path $\rho$. 
Then $t[n]=s[m]$.
\ \end{lemma} 

\begin{proof} We let \ $\rho=\rho_{0}\rho_{1}\rho_{2}\cdots \rho_{j}$,
where each of the $\rho_{i}$ is a subscripted basic operation symbol.
We will prove the lemma by induction on $j$. Our basis is when $j=0$,
making the operation \ $\|m,\rho_{0},n\|$. \ Assume that $\rho_{0}$ is
$f_{i}$ and $f$ is $k$-ary, so $s$ is 
$f(u_{1},\dots u_{i-1},t,u_{i+1},\dots u_{k})$,
where the $u_{j}$ are terms. The only assignment to $t[n]$ is then
made by $f[n]:=x_{i}[m]$\ , \ giving \ $t[n]=f[n]=x_{i}[m]=s[m]$,
\ as desired. For the induction step, assume the statement is
true for \ $j-1$, \ and that we want to show it for the path \ 
$\rho=\rho_{0}\rho_{1}\rho_{2}\cdots \rho_{j}$.
\  We write \ $\|m,\rho,n\|$ \ as \ 
$\|m,\rho_{j},b\|+\|b,\rho_{0}\rho_{1}\rho_{2}\cdots \rho_{j-1},n\|$
\ for some new index $b$, and let $\psi$ be \ 
$\rho_{0}\rho_{1}\cdots \rho_{j-1}$,
\ so \ $\rho=\psi\rho_{j}$. \ 
By the statement for \ $j-1$, \ $t[n]=u[b]$,
where $u$ is the subterm of $s$ with path $\psi$. \ We assume $\rho_{j}$
is $f_{i}$ , and have $u=f(v_{1},\dots v_{i-1},t,v_{i+1},\dots v_{k})$,\ where
$f$ is $k$-ary and the $v_{j}$ are terms. Then $u[b]=f(v_{1},\dots v_{i-1},t,v_{i+1},\dots v_{k})[b]=s[m]$,
since indices are chosen to minimize collisions and $\|m,\rho_{j},b\|$
is the only summand that assigns a value to the $b$-th component. Thus
\ $t[n]=u[b]=s[m]$, as desired. 
\end{proof} 

Given the algebra
operation $\|m,\rho,n\|$, we define the {\em tweaked} operation
$\|m,\rho,n\|'$ to be identical to $\|m,\rho,n\|$ except for
one assignment. Writing $\rho$ as $\psi\rho_{j}$, we assume that $\rho_{j}$
is $f_{i}$, for some basic operation $f$ and some $i$ less than or 
equal to the arity of $f$.  Then $\|m,\rho,n\|$ \ has an assignment of
the form \ $f[b] := x_{i}[m]$. We modify it by adding $1$, giving
\ $f[b] := (x_i[m]+1) \mod2$ in the definition of $\|m,\rho,n\|'$,
and keeping the rest of $\|m,\rho,n\|$ unchanged. 
A slight modification of the
proof of the previous lemma then establishes the following. 

\begin{lemma}\label{tweaked lemma}
Let the tweaked algebra operation $\|m,\rho,n\|'$ be duplicate
free, and let $t$ be an algebra term and let $s$ be a subterm of $t$,
where $\rho$ is the path of $s$ in $t$.  Then $t[n] = (s[m]+1) \mod2$.
\end{lemma}

\end{section}

\begin{section}{Algebras separating terms}   

%Define $\Alg(L)$

\begin{lemma}\label{no s or t lemma}
If the equivalence class of $s$ and $t$ under $\equiv$ is larger than $\{ s,t \}$,
then  $s$ and $t$ are separated in a finite algebra.
\end{lemma}

\begin{proof}
Let $s$ and $t$ be given, and let $\equiv$ be the relation on terms generated from
$s \equiv t$.  Note that only terms that are subterms of $s$ or $t$ are 
equivalent to terms they are not equal to.  (This is easily proved by induction on 
the length of deductions.)

If $s$ is a proper subterm of $t$ or $t$ is a proper subterm of $s$, then we can 
separate $s$ and $t$ in a finite algebra as follows.  Without loss of generality, 
assume $s$ is a proper subterm of $t$, and let $\rho$ be the path of $s$ in $t$.
Then in $\Alg( \| 0,\rho,0 \|' )$, we have $t[0] = s[0] + 1$ by 
\ref{tweaked lemma}, showing that $s$ and $t$
are separated.  If $s = t$, then all the $\equiv$ classes are singletons, and the 
lemma also holds.

So assume neither $s$ nor $t$ is a subterm of the other, and let $E$ be the equivalence
class of $s$ and $t$ under $\equiv$.  Suppose 
$E$ is larger than $\{s,t\}$, so there is some $r \in E$ with 
$r \neq s$ and $r \neq t$, where we assume that the deduction $D$ of $r \equiv s$ or
$r \equiv t$ that shows this is of minimal length for all deductions giving 
$r \equiv s$ or $r \equiv t$ for some term $r$ in $E$ other than $s$ or $t$.

Now $D$ can not be the deduction $\langle s \equiv t \rangle$, so the deduction must end 
with $s \equiv r$ or $t \equiv r$, which is obtained by applying a deduction rule.  This 
rule must be {\em Transitive} or {\em Decompose}, since the other two rules only
produce $\mathbf{False}$.

If $r \in E$ is obtained by an application 
of {\em Transitive}, we assume without loss of generality that $s \equiv r$ is deduced,
and that the chain of equivalences used is as short as possible.
Then for some $n \geq 3$ there are $p_1, \dots p_n$ with 
$s = p_1$, $r = p_n$, and 
where the statements $p_1 \equiv p_2$, $p_2 \equiv p_3$, and so on are already deduced,
where none of these statements is trivial.  If $p_2$ is $t$, then the chain 
$p_2 \equiv p_3 \equiv \dots p_n = r$ is shorter than the original one, contradicting
our assumption.  Thus $p_2$ must not be in $\{ s,t \}$, contradicting our assumption that 
the deduction was of minimal length.

In the case where $r \in E$ is obtained by an application 
of {\em Decompose}, we assume without loss of generality that $s \equiv r$ is deduced
from some earlier statement $p \equiv q$, where $s$ is a proper subterm of $p$ and
$r$ is a proper subterm of $q$.  Now $p$ can not be a subterm of $s$, for then
$s$ would be a proper subterm of itself.  Thus $p$ must be a subterm of $t$, making $s$ a
subterm of $t$, which is also a contradiction.
\end{proof}

In view of this lemma, we will henceforth assume that the only statement involving
$s$ or $t$ that can be deduced is $s \equiv t$ itself.

Another special case is dealt with by the following lemma.

\begin{lemma}\label{s and t not variable lemma}
If $s$ or $t$ is a variable and $s$ and $t$ can not be unified, 
then  $s$ and $t$ are separated in a finite algebra.
\end{lemma}

\begin{proof}
Without loss of generality, assume that $t$ is the variable $x$.
If $x$ does not appear in $s$, then $s$ and $t$ can be unified by substituting
$s$ for $x$.  So assume that $x$ occurs in $s$, and let $\rho$ be the path of
this occurrence.  Then $s$ and $t$ are separated in 
$\Alg( \| 0,\rho,0 \|')$,
for $s[0] = x[0] + 1$ and $t[0] = x[0]$ there.
\end{proof}

We will henceforth assume that neither $s$ nor $t$ is a variable.  This is important,
because components of non-variable subterms are $0$ everywhere in $\Alg(L)$, 
except where values are explicitly assigned to them by summands of $L$.

%Given a sum of components of terms $t_0[a_0] + t_1[a_1] + \dots t_k[a_k]$, we write 
%$(t_0[a_0] + \dots t_k[a_k]) / \Alg(L)$ for the value of 
%$$t_0[a_0] + t_1[a_1] + \dots t_k[a_k]$ in the finite algebra $\Alg(L)$,
%provided that this value is constantly $0$ or constantly $1$.  If the value
%of $t_0[a_0] + t_1[a_1] + \dots t_k[a_k]$ is not constant in $\Alg(L)$, we 
%say that $(t_0[a_0] + \dots t_k[a_k]) / \Alg(L)$ is undefined.  
%Note that for any
%sums of transformations $L$ and $L'$ where $(t_0[a_0] + \dots t_k[a_k]) / \Alg(L)$
%and $(t_0[a_0] + \dots t_k[a_k]) / \Alg(L')$ are defined, we have 
%$(t_0[a_0] + \dots t_k[a_k]) / \Alg(L+L') = (t_0[a_0] + \dots t_k[a_k]) / \Alg(L)
%+ (t_0[a_0] + \dots t_k[a_k]) / \Alg(L)$.

\begin{lemma}\label{sums for k and 0 lemma}
If $p \equiv q$ can be deduced from $s \equiv t$, then there is an index $k$ such that
there is a finite sum of transformations $L$ so that $p[k] + q[k] = s[0] + t[0]$
in $\Alg(L)$.  If $p \equiv q$ is $s \equiv t$, then $k$ must be $0$, and if 
$p \equiv q$ is not $s \equiv t$, then $k$ can have any nonzero value.
\end{lemma}

\begin{proof}
We prove this by induction on the minimal length of a deduction 
of $p \equiv q$.  Our basis is where
$p \equiv q$ is $s \equiv t$.  In this case we let $L$ be the zero vector and take $k = 0$.

For the induction step, let $p \equiv q$ be different from $s \equiv t$, and let 
$D$ be a minimal deduction of $p \equiv q$.  Then $p \equiv q$ must be deduced 
in the last step of $D$ by using either {\em Decompose} or {\em Transitive}, giving us two cases.

Assume $p \equiv q$ is deduced by {\em Decompose} from the statement
$u(\dots, p, \dots) \equiv u(\dots, q, \dots)$, where $u$ is some term such that
$p$ and $q$ have the same path $\rho$ in $u$. 
If $u(\dots, p, \dots) \equiv u(\dots, q, \dots)$ is $s \equiv t$, we let 
$k \neq 0$ be given and take $L$ to be $\| k,\rho,0 \|$.  Then 
$p[k] + q[k] = s[0] + t[0]$ by Lemma \ref{path lemma}.

So assume $u(\dots, p, \dots) \equiv u(\dots, q, \dots)$ is not $s \equiv t$,
and let any nonzero index $k$ be given.  By our induction hypothesis, 
there is a nonzero index $m \neq k$
and a sum of transformations $L'$ with 
$u(\dots, p, \dots)[m] + u(\dots, q, \dots)[m] = s[0] + t[0]$ in $\Alg(L')$.
We now let $L$ be $L' + \| k, \rho,m \|$, and consider the value of 
$p[k] + q[k]$ in $\Alg(L)$.  We have 
$p[k] + q[k] = u(\dots, p, \dots)[m] + u(\dots, q, \dots)[m]$ 
there, since this is true in $\Alg(\| k, \rho,m \|)$ and summands of 
$L'$ make no assignments to $k$-th or $m$-th components.
We also have $u(\dots, p, \dots)[m] + u(\dots, q, \dots)[m] = s[0] + t[0]$
in $\Alg(L)$, since this is true in $\Alg(L')$ and because $\| k, \rho,m \|$ makes no assignments to $0$-th components.  
So $p[k] + q[k]  = u(\dots, p, \dots)[m] + u(\dots, q, \dots)[m]) =
s[0] + t[0])$ in $\Alg(L)$.

Now assume $p \equiv q$ is deduced  from
$p_1 \equiv p_2, p_2 \equiv p_3, \dots p_m \equiv p_{m+1}$
by {\em Transitive}, where 
$p_1$ is $p$, $p_{m+1}$ is $q$, and $p_i \equiv p_{i+1}$ is in $D$ for all $i$.
By our assumption after Lemma \ref{no s or t lemma}, we have that none of
the statements $p_i \equiv p_{i+1}$ is $s \equiv t$.  Thus given any nonzero
index $k$, we have the following sums of terms for $1 \leq i \leq m$.
We let $L_i$ be such that $p_i[k] + p_{i+1}[k] = s[0] + t[0]$ in 
$\Alg(L_i)$.  As usual, we assume when $i \neq j$ that $L_i$ and $L_j$ 
have no indices other than $0$ and $k$ that they both reference.
Letting $L = L_1 + L_2 + \dots L_{m}$, we have in $\Alg(L)$
that $s[0] + t[0]$ is equal to the sum
$(p_1[k] + p_2[k]) + (p_2[k] + p_3[k]) + \dots (p_{m}[k] + p_{m+1}[k])$.
This is because none of the $L_i$ assigns values to $k$-th components, so 
the only interaction between terms in the various $L_i$ will be at  
$0$-th components, where they assign values to $s[0]$ and $t[0]$.
We are working modulo $2$, so all of the middle terms cancel, 
giving  $s[0] + t[0] =
p_1[k] + p_2[k] + p_2[k] + p_3[k] + \dots p_{m}[k] + p_{m+1}[k] =
p_1[k] + p_{m+1}[k] = p[k] + q[k]$ in $\Alg(L)$.
\end{proof}

\begin{theorem}\label{separation theorem} Let $s$ and
$t$ be terms which can not be unified.  Then $s$ and $t$
are separated in a finite algebra. \end{theorem} 

\begin{proof}
Suppose that $s$ and $t$ can not be unified, and let $\equiv$
be the closure of $s \equiv t$ under the unification axiom and
transitivity.  The relation $\equiv$ must either fail to be homogeneous
or fail to be acyclic.

\textbf{Case 1:}  Suppose $\equiv$ is not homogeneous.  

Then there is a minimal deduction
$D$ of $f(\dots) \equiv g(\dots)$, where $f$ and $g$ are different 
operation symbols.  By Lemma \ref{sums for k and 0 lemma}, we have a
sum of transformations $L'$ and an index $k$ so that 
$f(\dots)[k] + g(\dots)[k] = s[0] + t[0]$ in $\Alg(L')$.
Now consider the transformation we will call $\| f,k \|'$, which sets 
the $k$-th component of its output equal to $1$ if the basic operation
is $f$, and otherwise sets it to the default value of $0$. 
Thus adding the transformation $\| f,k \|'$ to $L'$ 
adds $1$ to the $k$-th
component of the output of $f$, and has no other effect.
We let $L$ be $L' + \| f,k \|'$, and claim that $s[0] \neq t[0]$
in $\Alg(L)$, or equivalently, that $s[0] + t[0] = 1$.

If $f(\ldots) \equiv g(\ldots)$ is $s \equiv t$, then we have 
$k = 0$ and $L = \| f,k \|' = \| f,0 \|'$.
Exactly one of $s$ or $t$ is of the form $f(\ldots)$, without loss of generality we assume that $s$ is.  Then $s[0] = 1$ and $t[0] = 0$, 
giving $s[0] + t[0] = 1$.

Now assume that $f(\ldots) \equiv g(\ldots)$ is not $s \equiv t$,
so $k \neq 0$.  The only summand of $L$ that assigns a value to
the $k$-th component is $\| f,k \|'$, so $f(\ldots)[k] = 1$ and
$g(\ldots)[k] = 0$.  Thus 
$s[0] + t[0] = f(\ldots)[k] + g(\ldots)[k] = 1 + 0 = 1$, as desired.

\textbf{Case 2:} Suppose that $\equiv$ is not acyclic.  

Then writing 
$\triangleleft$ for the relation ``is a subterm of'', there is 
a chain of subterms of $s$ or $t$ with 
$p_1 \equiv q_1 \triangleleft p_2 \equiv q_2 \triangleleft p_3 \dots
q_{m-1} \triangleleft p_{m+1} = p_1$, where $m \geq 1$ and at least
one of the subterm relationships is strict.  

For each $i$, the subterm 
relationship $q_i \triangleleft p_{i+1}$ gives us
that we have some term $u_i \in \{s,t\}$, and paths $\rho$ and $\sigma$,
so that $q_i$ has path $\rho\sigma$ in $u_i$ and 
$p_{i+1}$ has path $\rho$ in $u_i$.

We may assume that 
we have the shortest such chain that shows $\equiv$ is not acyclic,
and fix $m$ as this minimal chain length.
With this assumption, all of the subterm relationships in the chain
must be strict.  This is obviously true if $m = 1$, so consider the
case where $m > 1$, and suppose $q_i = p_{i+1}$ for some $i$.  Then 
we have $p_i \equiv q_i = p_{i+1} \equiv q_{i+1}$.  Since $\equiv$
is transitive, $p_i \equiv q_{i+1}$, so we can shorten the chain, a
contradiction.

With $m$ as above, we will also assume that our chain of subterms 
$p_1 \equiv q_1 \triangleleft p_2 \equiv q_2 \triangleleft p_3 \dots
q_{m} \triangleleft p_{m+1} = p_0$ is such that the sum of the 
lengths of the paths of the $q_i$ is as large as possible.
Specifically, we let $\length(\psi)$ denote the length of any path
$\psi$, and for each $i$ let $\rho_i$ be the path of $q_i$ in
whichever of $s$ or $t$ it is a subterm of.
We are then assuming that our 
chain of length $m$ showing $\equiv$ is not acyclic is such that 
$\length(\rho_1) + \length(\rho_2) + \dots \length(\rho_{m})$
is as large as possible for any such chain of length $m$.

Since $p_1 \equiv q_1$, $p_2 \equiv q_2$, and so on up to 
$p_{m} \equiv q_{m}$, Lemma \ref{sums for k and 0 lemma}
gives us non-zero indices $k_1, k_2, \dots k_{m}$ 
and sums of transformations $L_1, L_2, \dots L_{m}$ where
for all $i$ with $1 \leq i \leq m$ we have
$p_i[k_i] + q_i[k_i] = s[0] + t[0]$ in $\Alg(L_i)$.
As usual, we are assuming that $L_i$ and $L_j$ only have the index
$0$ that they both reference when $i \neq j$.

Since $q_i \triangleleft p_{i+1}$ for $1 \leq i \leq m$,
then for each $i$ the subterm $q_i$ has path $\sigma_i$ 
in $p_{i+1}$.  Since all of the subterm relations are strict, 
none of the $\sigma_i$ 
are $\Lambda$.  Now we let $L$ be 
$L_1 + L_2 + \dots L_{m} + \| k_1,\sigma_1,k_2 \| + 
\| k_2,\sigma_2, k_3 \| + \dots
\| k_{m-1}, \sigma_{m-1},k_m \| + \| k_{m},\sigma_{m},k_1 \|'$.

Since none of $\| k_1,\sigma_1,k_2 \| + 
\| k_2,\sigma_2, k_3 \| + \dots
\| k_{m-1}, \sigma_{m-1},k_m \| + \| k_{m},\sigma_{m},k_1 \|'$ 
assign values to $0$-th 
components, the value of $s[0] + t[0]$ comes from 
$L_1 + L_2 + \dots L_{m}$, and since none of the $L_i$ share
indices other than $0$, we get $s[0] + t[0]$ by adding together the
$p_i[k_i] + q_i[k_i]$, giving
$s[0] + t[0] = p_1[k_1] + q_1[k_1] + \dots 
p_{m}[k_{m}] + q_{m}[k_{m}]$.

Now consider any $p_i[k_i] + q_i[k_i]$ for $1 \leq i \leq m$,
where we interpret $i-1$ as $m$ when $i = 1$.
For simplicity, we will work with $\| k_{m},\sigma_{m},k_1 \|$
instead of $\| k_{m},\sigma_{m},k_1 \|'$ for the moment,
and switch to using $\| k_{m},\sigma_{m},k_1 \|'$ 
at the end of the proof.
The only summand of $L$ that assigns to $k_i$-th components is
$\| k_{i-1},\sigma_{i-1}, k_i \|$.  This could in principle give 
both $p_i[k_i]$ and $q_i[k_i]$ nonzero values.  We claim that only
$p_i[k_i]$ is given a nonzero value by
$\| k_{i-1},\sigma_{i-1}, k_i \|$.

Since $q_{i-1}$ is a subterm of $p_i$ with path $\sigma_{i-1}$
for $2 \leq i \leq m$, we have $p_i[k_i] = q_{i-1}[k_{i-1}]$.  Assume 
$q_i[k_i]$ is given a nonzero value.  This must be because $q_i$ 
has a subterm $r$ with path $\sigma_{i-1}$ in $q_i$.  In this case
we have $p_i \equiv q_i$, that $q_{i-1}$ is a subterm of $p_i$ with
path $\sigma_{i-1}$, and that $r$ is a subterm of $q_i$ with path
$\sigma_{i-1}$.  Thus $q_{i-1} \equiv r$ by {\em Decompose},
which also gives us $p_{i-1} \equiv r$ by {\em Transitive}.
Now $r$ is a subterm of $q_i$ and hence a subterm of $p_{i+1}$,
so we have the chain 
$p_1 \equiv q_1 \triangleleft \dots 
p_{i-1} \equiv r \triangleleft p_{i+1} \dots
q_{m} \triangleleft p_{m+1} = p_1$.
If $m$ is greater than $1$, this is a shorter chain, which contradicts 
our assumption.  So assume $m = 1$, which makes $p_i = p_1$ and $q_i = q_1$
for all $i$.  Then our original chain was $p_1 \equiv q_1 \triangleleft p_1$,
but we also have the chain $p_1 \equiv r \triangleleft p_1$ where 
$r$ is a proper subterm of $q_1$, contradicting our assumption that the 
sum of the lengths of the paths to the $q_i$ is as large as possible.

Thus we have 
$p_i[k_i] + q_i[k_i] = q_{i-1}[k_{i-1}] + q_i[k_i]$ for 
$2 \leq i \leq m$.
For $i = 1$ the argument is the same, except that 
$\| k_m,\sigma_m,k_1 \|'$ makes 
$p_1[k_1] + q_1[k_1] = q_m[k_m] + 1 + q_1[k_1]$.
Thus we have
$s[0] + t[0] = (p_1[k_1] + q_1[k_1]) + \dots (p_{m}[k_{m}] + q_{m}[k_{m}])
= (q_{m}[k_{m}] + q_1[k_1] + 1) + (q_1[k_1] + q_2[k_2]) +
\dots (q_{m-1}[k_{m-1}] + q_{m}[k_{m}]) =
2(q_1[k_1] + q_2[k_2] + \dots q_{m}[k_{m}]) + 1 = 1$, since we calculate 
component values modulo $2$.  This shows $s[0]$ is always not equal to $t[0]$
in $\Alg(L)$.
\end{proof} 

As a corollary, we have the following.

\begin{corollary}\label{finite model corollary} 
Whenever the first order sentence $\sigma$ is a universally quantified conjunction
of negated atomic formulas, $\sigma$ is consistent iff it has a finite
model. \end{corollary} 

\begin{proof}
If $\sigma$ has a model, it is consistent.  So assume $\sigma$ is 
consistent.  We have that $\sigma$ is of the form
$(\forall v_1) \dots (\forall v_k)(\neg(\theta_1) \wedge \dots \neg(\theta_n))$,
where each of the $\theta_i$ is an atomic formula.  If $\theta_i$ is of the 
form $R(\dots)$ where $R$ is a relation symbol, we will just interpret $R$ as
always false in our finite model.  This leaves us with the case where
$\theta_i$ is $s_i(v_1, \dots v_k) = t_i(v_1, \dots v_k)$ for terms
$s_i$ and $t_i$.  If $s_i$ and $t_i$ can be unified, then $\neg(\theta_i)$
is not always true
and hence $\sigma$ has no models, and $\sigma$ is not consistent.  So 
$s_i$ and $t_i$ can not be unified, and the theorem gives a finite model 
$A_i$ where $\neg(\theta_i)$ is always true.  
Now we take the product of all these $A_i$,
and interpret all relations in the language as $\mathbf{False}$ in the product, giving a finite model of $\sigma$.
\end{proof}

\end{section}

%\centerline{}\textbf{Addresses}\vspace{0.5em}

%David Hobby. Department of Mathematics, State
%University of New York, New Paltz NY 12561 -- U.S.A.

%\underline{Email}: \ hobbyd@newpaltz.edu \ \vspace{4em}

\end{document}